\documentclass[a4paper,11pt,twoside]{amsart}

\usepackage[latin1]{inputenc}
\usepackage{amssymb,amsthm}
\usepackage[all]{xy}
\usepackage{amsmath}
\usepackage{indentfirst}
\usepackage{fancyhdr}
\usepackage{amsfonts}
\usepackage{textcomp}
\usepackage{xypic}
\usepackage{graphicx}
\usepackage{enumerate}
\usepackage{verbatim}
\usepackage{hyperref}
\usepackage{amsthm}
\usepackage{eucal}
\usepackage{amssymb}
\usepackage{mathrsfs}

\newtheorem{teo}{Theorem}[section]
\newtheorem{pro}[teo]{Proposition}
\newtheorem{lem}[teo]{Lemma}
\newtheorem{co}[teo]{Corollary}

\theoremstyle{definition}
\newtheorem{de}[teo]{Definition}

\newtheorem{no}[teo]{Notation}

\newtheorem{re}[teo]{Remark}

\newcommand{\Q}{\mathbb{Q}}
\newcommand{\R}{\mathbb{R}}

\newcommand{\N}{\mathbb{N}}
\newcommand{\OO}{\mathcal{O}}
\newcommand{\I}{\mathcal{I}}

\theoremstyle{remark}
\newtheorem{rem}{Remark}[teo]
\newtheorem{cla}[rem]{Claim}

\title{Discrepancies of non-$\Q$-Gorenstein varieties}
\author{Stefano Urbinati}
\date{}

\begin{document}

\address{Department of Mathematics, University of Utah, 155 South 1400 East, Salt Lake City, UT 84112, USA}
\email{urbinati@math.utah.edu}



\begin{abstract}
We give an example of a non $\Q$-Gorenstein variety whose canonical divisor has an irrational valuation and example of a non  $\Q$-Gorenstein variety which is canonical but not klt. We also give an example of an irrational jumping number and we prove that there are no accumulation points for the jumping numbers of normal non-$\Q$-Gorenstein varieties with isolated singularities.
\end{abstract}

\maketitle

\section{Introduction}\label{sec:intro}

The aim of this paper is to investigate some surprising features of singularities of normal varieties in the non-$\Q$-Gorenstein case as defined by T. de Fernex and C. D. Hacon (cf. \cite{MR2501423}). In that paper the authors focus on the difficulties of extending some invariants of singularities in the case that the canonical divisor is not $\Q$-Cartier. Instead of the classical approach where we modify the canonical divisor by adding a boundary, an effective $\Q$-divisor $\Delta$ such that $K_X + \Delta$ is $\Q$-Cartier, they introduce a notion of $pullback$ of (Weil) $\Q$-divisors which agrees with the usual one for $\Q$-Cartier $\Q$-divisors. In this way, for any birational morphism of normal varieties $f: Y \to X$, they are able to define relative canonical divisors $K_{Y/X}= K_Y +f^*(-K_X)$ and $K^-_{Y/X}= K_Y -f^*(K_X)$. The two definitions coincide when $K_X$ is $\Q$-Cartier and using $K_{Y/X}$ and $K_{Y/X}^-$ de Fernex and Hacon extended the definitions of canonical singularities, klt singularities and multiplier ideal sheaves to this more general context. 

In this setting some of the properties characterizing the usual notions of singularity (see \cite[Section 2.3]{MR1658959}) seem to fail due to the asymptotic nature of the definitions of the canonical divisors.

We focus on three properties that for $\Q$-Gorenstein varieties are straightforward:
\begin{itemize}
\item The relative canonical divisor always has rational valuations (cf. \cite[Theorem 92]{oai:arXiv.org:0809.2579}).

\item A canonical variety is always kawamata log terminal (cf. \cite[Definition 2.34]{MR1658959}).

\item The jumping numbers are a set of rational numbers that have no accumulation points (cf. \cite[Lemma 9.3.21]{MR2095472}).

\end{itemize}
In this article we investigate these properties for non-$\Q$-Gorenstein varieties.

In the third section, we show that if $X$ is klt in the sense of \cite{MR2501423}, then the relative canonical divisor has rational valuations and we give an example of a (non klt) variety $X$ with an irrational valuation and we use it to find an irrational jumping number (Theorem \ref{irrjump}).

In the fourth section we give an example of a variety with canonical but not klt singularities (Theorem \ref{canonklt}) and we prove that the finite generation of the canonical ring implies that the relative canonical model has canonical singularities (Proposition \ref{fgcr}). 

In the last section, using one of the main results in \cite{MR2501423}, namely that every effective pair $(X,Z)$ admits $m$-compatible boundaries for $m\geq 2$ (see Theorem \ref{teoTC} below), we show that for a normal variety whose singularities are either klt or isolated, it is never possible to have accumulation points for the jumping numbers (Theorem \ref{teojn}).

\bigskip

{\small\noindent {\bf Acknowledgements.} } The author would like to thank his supervisor C. D. Hacon for his support and the several helpful discussions and suggestions. The author would also like to thank E. Macr\`i and T. de Fernex for many useful comments. 

The author is grateful to the referee for many valuable suggestions.

\section{Basic definitions}\label{sec:basic}

 The following notations and definitions are taken from \cite{MR2501423}.

\begin{no} \label{noX} Throughout this paper $X$ will be a normal variety over the complex numbers.
\end{no}

Let us denote by $v = \mbox{val}_F$ a divisorial valuation on $X$ with respect to the prime divisor $F$ over $X$. Given a proper closed subscheme $Z \subset X$ we define $v(Z)$ as 
$$v(Z)= v(\I_Z):= \min\{v(\phi)| \phi \in \I_Z(U), U\cap c_X(v)\ne \emptyset \}$$
where $\I_Z \subset \OO_Z$ is the ideal sheaf of $Z$. The definition extends to $\R$-linear combinations of proper closed subschemes. The same definition works in a natural way for linear combinations of fractional ideal sheaves.

To any fractional ideal sheaf $\I$ on $X$, we associate the divisor $$\text{div}(\I) := \sum_{E \subset X} \text{val}_E(\I)\cdot E$$
where the sum is over all prime divisors $E$ on $X$ and val$_E$ denotes the divisorial valuation with respect to $E$.

\begin{de} Let $X$ be as in Notation \ref{noX}. The $\natural$-valuation (or natural valuation) along a valuation $v$ of a divisor $F$ on $X$ is 
$$v^{\natural}(F):= v(\OO_X(-F)).$$
\label{val} Let $D$ be a $\Q$-divisor on $X$. The valuation along $v$ of $D$ is 
$$v(D):= \lim_{k\to \infty} \frac{v^{\natural}(k!D)}{k!}= \inf_{k\geq 1} \frac{v^{\natural}(kD)}{k} \, \in  \, \R.$$

\end{de}

\begin{no} \label{noY} Let $X$ be as in Notation \ref{noX}. Let us consider a projective birational morphism $f: Y \to X$ from a normal variety $Y$. 
\end{no}
We have the following definitions:

\begin{de} \label{pullback} Using notation \ref{noY}, for any divisor $D$ on $X$, the $\natural$-pullback of $D$ to $Y$ is defined to be 
$$f^{\natural}D=\text{div}(\OO_X(-D)\cdot \OO_Y).$$
This is the natural choice to obtain a reflexive sheaf,  $\OO_Y(-f^{\natural}D)= (\OO_X(-D)\cdot \OO_Y)^{\vee \vee}$.

We also need a good definition of $pullback$ of $D$ to $Y$, that needs to coincide with the classical one when we restrict to non-singular varieties. We have:
$$f^*D:= \sum \text{val}_E(D)\cdot E,$$
where the sum is taken over all the prime divisors $E$ on $Y$.

\end{de}

We now give the main definitions that characterize multiplier ideal sheaves. 

\begin{de} Let $f: Y \to X$ be as in Notation \ref{noY}, for every $m \geq 1$, the $m$-th limiting relative canonical $\Q$-divisor $K_{m, Y/X}$ of $Y$ over $X$ is
$$K_{m,Y/X}:= K_Y - \frac{1}{m}\cdot f^{\natural}(mK_X).$$
The relative canonical $\R$-divisor $K_{Y/X}$ of $Y$ over $X$ is 
$$K_{Y/X}:= K_Y +  f^*(-K_X)$$
\end{de}

In particular $K_{m, Y/X}\leq K_{mq,Y/X}\leq K_{Y/X}$. Also, taking the limsup of the coefficients of the components of the $\Q$-divisor $K_{m, Y/X}$, one obtains the $\R$-divisor $K_{Y/X}^{-}:= K_Y - f^*K_X$ which satisfies $K_{Y/X}^{-} \leq K_{Y/X}$ (the two divisors coincide if $X$ is $\Q$-Gorenstein i.e. if $K_X$ is $\Q$-Cartier).
\vskip .3cm
Recall that an effective $\Q$-divisor $\Delta$ is a boundary on $X$ if $K_X + \Delta$ is a $\Q$-Cartier $\Q$-divisor.
\begin{de} Let $f: Y \to X$ as in Notation \ref{noY}, let $\Delta$ be a boundary on $X$ such that $K_X+ \Delta$ is $\Q$-Cartier, and let $\Delta_Y$ be the proper transform of $\Delta$ on $Y$. The log relative canonical $\Q$-divisor of $(Y, \Delta_Y)$ over $(X, \Delta)$ is given by:
$$K_{Y/X}^{\Delta}:= K_Y + \Delta_Y - f^*(K_X + \Delta)= K_Y + \Delta_Y + f^*(-K_X - \Delta).$$
\end{de}

In particular, for every boundary $\Delta$ on $X$ and every $m\geq 1$ such that $m(K_X + \Delta)$ is Cartier, we have 
$$K_{m, Y/X}= K_{Y/X}^{\Delta} - \frac{1}{m}\cdot f^{\natural}(-m\Delta) - \Delta_Y  \quad \text{and} \quad K_{Y/X}= K_{Y/X}^{\Delta} + f^*\Delta - \Delta_Y.$$
Note that $K_{Y/X}^{\Delta} \leq K_{m, Y/X} \leq K^-_{Y/X}$.
\begin{de} \label{pair}
Consider a pair $(X,I)$ where $X$ is a normal quasi-projective variety and $I= \sum a_k \I_k$ is a formal $\R$-linear combination of non-zero fractional ideal sheaves on $X$.  Let us denote by $Z=\sum a_k Z_k$ the associated subscheme, where $Z_k$ is the subscheme generated by $\I_k$.

We define a \emph{log resolution} of this pair as a proper birational morphism $f:Y \to X$, where $Y$ is a smooth variety, such that for every $k$:
\begin{itemize} 
\item The sheaf $\I_k\cdot\OO_Y$ is an invertible sheaf corresponding to a divisor $E_k$ on $Y$.
\item The exceptional locus Ex$(f)$ is a divisor.
\item The union of the supports of $E_k$ and Ex$(f)$ is simple normal crossing. 
\end{itemize}

If $\Delta$ is a boundary on $X$, then a log resolution for $((X,\Delta);I)$ is given by a resolution of $(X,I)$ such that Ex$(f)$,  $E$, Supp$f^*(K_X+\Delta)$ are divisors and their union Ex$(f)\cup E \cup$ Supp$f^*(K_X+\Delta)$ has simple normal crossings. 
\end{de}

\begin{de} Let $(X,Z)$ be as in Definition \ref{pair}. Let $f:Y \to X $ be a log resolution with $Y$ normal, and let $F$ denote a prime divisor on $Y$. For any integer $m \geq 1$, we define the $m$-th limiting log discrepancy of $(X,Z)$ along $F$ to be 
$$a_{m,F}(X;Z):=\text{ord}_F(K_{m, Y/X}) +1 -\text{val}_F(Z).$$
\end{de}

\begin{de} \label{dLTklt} Using the notation above, the pair $(X,Z)$ is said to be \emph{log terminal} if there is an integer $m_0$ such that $a_{m_0,F}(X,Z)>0$ for every prime divisor $F$ over $X$. 

We say that an effective pair is $klt$ if and only if there exists a boundary $\Delta$ such that $((X, \Delta); Z)$ is $klt$ (kawamata log terminal) in the usual sense.
\end{de}

 In particular the notions of log terminal and klt are equivalent because of the following (cf. \cite[Thm 5.4]{MR2501423}):
\begin{teo} \label{teoTC}Every effective pair $(X, Z)$ admits $m$-compatible boundaries for $m\geq 2$, where, a boundary $\Delta$ is said to be $m$-compatible if:
\begin{enumerate}[i)]
\item  $m \Delta$ is integral and $\lfloor \Delta \rfloor=0$.
\item No component of $\Delta$ is contained in the support of $Z$.
\item $f$ is a log resolution for the log pair $((X,\Delta); Z + \OO_X(-mK_X))$.
\item $K_{Y/X}^{\Delta}=K_{m, Y/X}$.
\end{enumerate} 
\end{teo}

\begin{no} \label{LTklt} Because of the previous Theorem and the genaral notation in the literature, now on we will abuse our notation saying that a normal variety $X$ is klt whenever it is log terminal according to Definition \ref{dLTklt}. A pair $(X, \Delta)$ will be klt in the usual sense.

\end{no}

\begin{re} \label{boundary} Since $K_{Y/X}^{\Delta} \leq K^-_{Y/X}$, by Theorem \ref{teoTC} and Definition $\ref{val}$
$$\mbox{val}_F(K_{Y/X}^-)= \sup \{\mbox{ord}_F(K_{Y/X}^{\Delta})| (X, \Delta) \mbox{ is a log pair}\}.$$ 
Note that we are considering a limit and hence we may have irrational valuations.
\end{re}
\begin{de} \label{can} Let $X$ be an in Notation \ref{noX}. Let $X' \to X$ be a proper birational morphism with $X'$ normal, and let $F$ be a prime divisor on $X'$. The \emph{log-discrepancy} of a prime divisor $F$ over $X$ with respect to $(X,Z)$ is 
$$a_F(X,Z):= \mbox{ord}_F(K_{X'/X}) + 1 - \mbox{val}_F(Z).$$
Using the notation in Definition \ref{pair}, the pair $(X,Z)$ is said to be \emph{canonical} (resp. \emph{terminal}) if $a_F(X, Z)\geq 1$ (resp. $>1$) for every exceptional prime divisor $F$ over $X$.
\end{de}

Recall that by \cite[Proposition 8.2]{MR2501423}, a normal variety $X$ is canonical if and only if for sufficiently divisible $m\geq 1$, and for every sufficiently high log resolution $f:Y \to X$ of $(X, \OO_X(mK_X))$, there is an inclusion 
$$ \OO_Y\cdot \OO_X(mK_X) \hookrightarrow \OO_Y(m K_Y).$$
We have the following useful lemma:
\begin{lem} Using Notation \ref{noX}, \label{can1} let $f:Y\to X$ be a proper birational morphism such that $Y$ is canonical. If $\OO_Y\cdot \OO_X(mK_X) \hookrightarrow \OO_Y(m K_Y)$ for sufficiently divisible $m\geq 1$, then $X$ is canonical.
\end{lem}

\begin{proof} Let $g:Y' \to X$  be a log resolution of  $(X, \OO_X(mK_X))$. Without loss of generality we can assume that $g$ factors through $f$, so that we have $h: Y' \to Y$ with $g= f \circ h$. In particular:
$$\OO_{Y'}\cdot \OO_X(mK_X) \hookrightarrow \OO_{Y'}\cdot \OO_Y(mK_Y)\hookrightarrow \OO_{Y'}(mK'_Y),$$
where the first inclusion is given by assumption and the second by $Y$ being canonical.
\end{proof}

\section{Irrational valuations}
Given a Weil $\R$-divisor $D$ on a normal variety $X$, we define the corresponding divisorial ring as $$\mathscr{R}_X(D):= \bigoplus_{m\geq 0} \OO_X (mD).$$ 
\begin{re} \label{fgcr} If $X$ is klt as in Notation \ref{LTklt}, then there exists $\Delta$ such that $(X, \Delta)$ is klt (in the usual sense) (cf. Theorem \ref{teoTC}). By \cite[Thm 92]{oai:arXiv.org:0809.2579} $\mathscr{R}_X(D)$ is finitely generated if and only if $D$ is a $\Q$-divisor.
\end{re}

\begin{pro} If a normal variety $X$ is klt, then for any prime divisor $F$ over $X$, the valuation $\mbox{val}_F(K^-_{Y/X})$ is rational, where $f:Y \to X$ is a projective birational morphism such that $F$ is a divisor on $Y$. 
\end{pro} 

\begin{proof} 
By Remark \ref{fgcr} and \cite[Lemma 2.1.6]{MR0217084}  we know that $\mathscr{R}_X(m_0 K_X)$ is generated by $\mathcal{O}_X(m_0 K_X)$ over $\mathcal{O}_X$ for some $m_0 >0$. It follows that $K^-_{Y/X}= K_{m_0, Y/X}$ and hence $\mbox{val}_F(K^-_{Y/X})= \mbox{val}_F(K_{m_0, Y/X})$ wich is a rational number (Remark \ref{boundary}).
\end{proof}

\vskip .3cm
Next we will construct an example of a threefold whose relative canonical divisor $K^-_{Y/X}$  has an irrational valuation. The example is given by the resolution of a cone singularity over an abelian surface. 

\vskip .3cm

Let us consider the abelian surface $X= E\times E$ where $E$ is an elliptic curve.
For this surface we have that $\overline{NE}(X)= \text{Nef}(X)\subset \text{N}^1(X)$, where $\mbox{N}^1(X)$ is generated by the classes $$f_1=[\{P\}\times E], \quad f_2=[E\times \{P\}], \quad \delta =[\Delta].$$
The intersection numbers are given by:
$$((f_1)^2)=((f_2)^2)=(\delta^2)=0 \quad \mbox{and} \quad (f_1.f_2)=(f_1.\delta)=(f_2.\delta)=1.$$
Given a class $\alpha = x f_1 + yf_2 + z\delta$ then $\alpha$ is nef if and only if $$xy + xz + yz \geq 0, \quad x+y+z \geq 0$$
and we obtain that $\text{Nef}(X)$ is a circular cone (cf. \cite[Ch II, Ex 4.16]{MR1440180}).

Next, we consider a double covering of this surface ramified over a general very ample divisor $H \in |2 \mathscr{L}|$ where $\mathscr{L}$ is an ample line bundle. This cover is given by $W = \text{Spec}_X(\OO_X \oplus \mathscr{L}^{^{\vee}}) $ with projection $p: W \to X$ induced by the inclusion $i: \OO_X \hookrightarrow \OO_X \oplus \mathscr{L}^{^{\vee}}$. In particular $$\omega_W= p^*(\omega_X \otimes \mathscr{L}).$$

There is an induced involution $\sigma: W \to W$. For any Cartier divisor $D$ on $W$,  $D + \sigma^*(D)$ is the pullback of a Cartier divisor on $X$. Since $H \in |2 \mathscr{L}|$ is general, the pullbacks of the generators $p^*f_i$ and $p^*\delta$ are irreducible curves on $W$. Since the map is finite, the pullback of an ample divisor (resp. nef, effective) on $X$ is ample (resp. nef, effective) on $W$.

It is easy to see that the map induced at the level of cones $p^*: \text{NE}(X)\to \text{NE}(W)$ is well defined, injective and
 
$$p^*\overline{\text{NE}}(X)= \overline{\text{NE}}(W)\cap p^*\mbox{N}^1(X).$$
Let us now consider any ample divisor $L$ on $X$ such that $p^*L$ defines an embedding $W \subset \mathbb{P}^n$. Let $C \subset \mathbb{P}^{n+1}$ be the projective cone over $W$. We want to investigate the properties of the relative canonical divisor. 

\begin{teo} \label{irrval} With the above construction, if $H \sim 6(f_1 +f_2)$, $L \sim (3f_1 +6f_2 +6\delta)$ and $f:Y \to C$ the blow up of the cone at the vertex, then the relative canonical divisor $K^-_{Y/X}$  has an irrational valuation.

\begin{proof} 
$p^*L$ defines an embedding, in fact $$p_*p^*L \cong  L \otimes (\OO_X \oplus \mathscr{L}^{\vee} )= L \oplus (L \otimes \mathscr{L}^{\vee}) \sim (3 f_1 + 6f_2 + 6 \delta) \oplus (3f_2 + 6\delta)$$ is a sum of very ample divisors. 
Since $f: Y \to C$ be the blow-up at the vertex, then $Y$ is isomorphic to the projective space bundle $\mathbb{P}(\OO_W\oplus \OO_W(p^*L))$, with the natural projection $\pi:Y \to W$. If we denote by $W_0$ the negative section we have $\OO_{W_0}(W_0) \cong \OO_W(-p^*L)$. Let us also denote by $W_{\infty}\sim W_0 + \pi^*p^*L$ the section at infinity. The canonical divisor $K_Y$ is given by $K_Y \sim \pi^*K_W - 2W_0 + \pi^*(-p^*L)$.

\begin{rem} Recall that we have an isomorphism $\text{Cl}W \cong \text{Cl}C$ defined by the map that associates to a divisor $D \subset W$ the cone over $D$,  $C_D \subset C$. A divisor $C_D$ is $\R$-Cartier if and only if $D \sim_{\R}k p^*L$, $k \in \R$. 
\end{rem}

We have that $K_C = f_*K_Y= C_{K_W} -C_{(p^*L)}$ and $C_{k(p^*L)}$ is an $\R$-Cartier divisor on $C$ such that $f^*(C_{k(p^*L)})= \pi^*(k(p^*L)) + kW_0$. Let $\Gamma$ be a boundary on $C$, then $\Gamma \equiv C_{\Delta}$ and since $K_C + \Gamma$ is $\Q$-Cartier we have that $K_C + C_{\Delta} = C_{K_W} - C_{(p^*L)} + C_{\Delta} \equiv C_{k(p^*L)}$ for some $k\in \Q$. In particular, given $s= k+1$, we have $s(p^*L) - K_W \equiv \Delta \geq 0$. So that
$$\Delta \equiv s(p^*L) - K_W \equiv s(p^*L) - \frac{1}{2}p^*H \equiv   p^*\left(s(3f_1 +6f_2 +6\delta) - \frac{1}{2}(6f_1 + 6f_2)\right).$$

By Remark \ref{boundary}, we have that:
$$\mbox{val}_{W_0}(K^-_{Y/C})= \sup \{\mbox{ord}_{W_0}(K_{Y/C}^{\Gamma})| (C, \Gamma) \mbox{ log pair}\}\geq \sup \{\mbox{ord}_{W_0}(K_{Y/C}^{C_{\Delta}})| (C, C_{\Delta}) \mbox{ log pair}\}.$$
Note that

$$\begin{aligned}
K_{Y/C}^{\Gamma} & \equiv  K_Y + f^{-1}_*\Gamma -f^*(K_{C}+\Gamma)   \equiv\\
&\equiv K_Y + f^{-1}_*\Gamma - f^*(C_{K_W} -C_{(p^*L)} + \Gamma) \equiv \\
&\equiv  \pi^*K_W -2W_0 + \pi^*(-p^*L) +f^{-1}_*\Gamma -\pi^*(K_{W} -p^*L + \Delta) - (s-1)W_0 \equiv \\
&\equiv  -(s+1)W_0 + f^{-1}_*\Gamma -\pi^*\Delta . 	
\end{aligned}$$

In particular $K_{Y/C}^{C_{\Delta}}= -(s +1)W_0$. Therefore if we let $t=\text{inf}\{s\in \R| \exists\, \Delta \geq 0, \; K_W + \Delta \equiv s(p^*L)\},$ then we have that
$$\mbox{val}_{W_0}(K^-_{Y/C})\geq -(1+t).$$

\begin{rem} \label{ampleb} Note that $\Delta$ is ample if $s>t$, in particular it is always possible to choose $\Delta= A/m$, with $A$ a smooth very ample Cartier divisor.
\end{rem}

\begin{cla}
 ${\rm val}_{W_0}(K^-_{Y/C})=-(1+t)$
\end{cla}
\begin{proof}  Let us consider any effective boundary $\Gamma \geq 0$. It sufficies to show that, in the previous construction, it is always possible to choose a boundary $\Delta \equiv s(p^*L) - K_W \subseteq W$ such that  $\mbox{ord}_{W_0}K_{Y/C}^{\Gamma} = \mbox{ord}_{W_0}K_{Y/C}^{C_{\Delta}}$. If  $f^*(K_C + \Gamma)= K_Y + f^{-1}_*\Gamma + kW_0$, let $\Delta = f_*^{-1}\Gamma|_{W_0} \geq 0$. Note that
$$\Delta =f_*^{-1}\Gamma|_{W_0}  \equiv -(K_Y + kW_0)|_{W_0} \equiv - K_W + (k-1) p^*L.$$
By what we have seen above (with $s = k-1$),  $ K_{Y/C}^{C_{\Delta}} = -k W_0$. Hence $\mbox{ord}_{W_0}K_{Y/C}^{\Gamma} = \mbox{ord}_{W_0}K_{Y/C}^{C_{\Delta}}$.
\end{proof}

We now return to the proof of Theorem \ref{irrval}.

Since $p^*\overline{\text{NE}}(X)= \overline{\text{NE}}(W)\cap p^*\mbox{N}^1(X)$ and $\Delta \geq 0$, the sum of the coefficients of $p^*(f_1)$, $p^*(f_2)$ and $p^*\delta$ has to  be positive, so that $s\geq \frac{2}{5}$. Again, because of the above isomorphism of cones, we have that $\Delta$ is effective if and only if it is nef:
 $$(\Delta^2)/4 = 9(8s^2 -7s + 1) \geq 0 \quad \Leftrightarrow \quad s \geq \frac{7 + \sqrt{17}}{16}\;\left(>\frac{2}{5}\right)$$
and we obtain an irrational valuation of the relative canonical divisor:
$$\mbox{val}_{W_0}(K^-_{Y/C})= - \frac{23 + \sqrt{17}}{16}.$$ 
\end{proof}
\end{teo}
\vskip .3cm
Using the result of Theorem \ref{irrval}, we now give an example of an irrational jumping number.
The following are the definitions of multiplier ideal sheaf and jumping numbers in the sense of \cite{MR2501423}.

\begin{de} As in Definition \ref{pair}, let $(X, Z)$ be an effective pair. The multiplier ideal sheaf of $(X,Z)$, denoted by $\I(X,Z)$, is the unique maximal element of $\{\I_m(X, Z)\}_{m \geq 1}$, where 
$$\I_m(X, Z):=f_{m_*}\OO_{Y_m}(\left\lceil K_{m,Y_m/X} - f_m^{-1}(Z)\right\rceil),$$ 
with $f_m:Y_m \to X$ a log resolution of the pair $(X, Z + \OO_X(-mK_X))$.
\end{de}

\begin{de} A number $\mu \in \R_{>0}$ is a \emph{jumping number} of an effective pair $(X, Z)$ if $\I(X, \lambda\cdot Z)\neq \I(X, \mu \cdot Z)$ for all $0 \leq \lambda < \mu$.
\end{de}
A relevant feature of the jumping numbers in the $\Q$-Gorenstein case is that they are always rational. 

\begin{teo} \label{irrjump} With the same construction as in Theorem \ref{irrval}, there are irrational jumping numbers for the pair $(C, P)$, where $P$ is the vertex of the projective cone.
\end{teo}
\begin{proof} We are considering $Z = P \subset C$ the vertex of the projective cone. Let us denote by $\mbox{Bl}_PC := f: Y \to C$ the blow-up of the vertex. Then we have that $f^{-1}(k\cdot Z)= k \cdot W_0$. By Theorem \ref{teoTC}, for every $m\geq 1$, there exists an $m$-compatible bounday $\Gamma_m$ such that $K_{m,Y/X}=K^{\Gamma_m}_{Y/X}$ and in particular $\I_m(X,Z)= \I((X,\Gamma_m);Z)$, hence $$\I(X, k\cdot Z) = \bigcup_m \I_m(X,k \cdot Z)=\bigcup_{\Gamma_m}\I((X,\Gamma_m);k \cdot Z).$$
Also, because of Remark \ref{ampleb}, the blow up is a log resolution of $((X, \Gamma_m); Z)$ for every $m\geq 1$, so that 
$$ \I((X,\Gamma_m);k \cdot Z) =f_*\OO_Y \left(\left\lceil K^{\Gamma_m}_{Y/X} - k \cdot W_0 \right\rceil\right)$$ 
and we conclude that we can compute the jumping numbers just considering the log resolution given by the blow-up $Y\to C$. We have
$$\I(X, k\cdot Z) = \bigcup_{\Gamma_m} f_*\OO_Y \left(\left\lceil K^{\Gamma_m}_{Y/X} - k \cdot W_0 \right\rceil\right).$$

Since $\mbox{val}_{W_0}(K^-_{Y/X})= -\frac{23 + \sqrt{17}}{16}$, the jumping numbers are of the form $k= t - \frac{23 + \sqrt{17}}{16}$ with $t$ any integer $\geq 1$.
\end{proof}

\section{Canonical singularities}

We begin by giving an example of a canonical singularity which is not klt.

Let us consider a construction similar to the one in the previous section. Let $S= \mathbb{P}^1\times \mathcal{E}$, where $\mathcal{E}$ is an elliptic curve. The canonical sheaf is 
$$\omega_S \sim \OO_{\mathbb{P}^1}(-2)\boxtimes \OO_{\mathcal{E}}. $$ 
Let $\mathscr{A}$ be an ample line bundle on $\mathscr{E}$ and let us consider the embedding $S \subseteq \mathbb{P}^n$ given by the very ample divisor $L= \OO_{\mathbb{P}^1}(2)\boxtimes \mathscr{A}^{\otimes 2}$. Let $C \subseteq \mathbb{P}^{n+1}$ the projective cone over $S$. 

\begin{teo} \label{canonklt} With the above construction, the singularity of $C$ at its vertex is canonical but not klt.
\end{teo}
\begin{proof}
With the same computation as in Theorem \ref{irrval}, let $f: Y \to C$ be the blow up of the cone at the origin $P$, $\pi: Y \to S$ the natural projection  and let us denote by $S_0$ the negative section. The canonical divisor $K_Y$ is given by $K_Y \sim \pi^*(K_S) -2S_0 +\pi^*(-L)$.
Let us compute $s$ in this case. We have $\Delta \equiv sL - K_S \sim \OO_{\mathbb{P}^1}(2s+2)\boxtimes \mathscr{A}^{\otimes 2s}$. In particular $\Delta$ is effective if and only if $s > 0 $. Hence, we  have:
$$\mbox{val}_{S_0}(K^-_{Y/C})=-1.$$
In particular $C$ is not klt. 

With a similar computation we will show that $C$ has canonical singularities. The relative canonical divisor used to characterize this type of singularities is $K_{Y/X}= K_Y + f^*(-K_X)$ and, by the notion of pullback given in Definition \ref{pullback}, it is given by an approximation of the form:
$$K^+_{m,Y/X}= K_Y + \frac{1}{m}f^{\natural}(-mK_X)$$
where, in this new definition, we have $K^+_{m, Y/X} \geq K^+_{mq,Y/X} \geq K_{Y/X}$. 
In particular the proof of the existence of an $m$-compatible boundary given in \cite{MR2501423} works also in this case with small modifications.

We now introduce the following corollary of Lemma \ref{can1}:
\begin{pro} \label{positive} Let $f:Y\to X$ be a proper birational morphism such that $Y$ is canonical. If ${\rm val}_F(K_{Y/X})\geq 0$ for all divisors $F$ on $Y$, then $X$ is canonical.
\end{pro}
\begin{proof} For all sufficiently divisible $m\geq 1$, $\mbox{val}_F(K^+_{m, Y/X})\geq 0$ (i.e. $mK_Y \geq -f^{\natural}(-mK_X)$), so that:
$$\OO_Y\cdot \OO_X(mK_X) \hookrightarrow (\OO_Y\cdot \OO_X(mK_X))^{\vee \vee} = \OO_Y(-f^{\natural}(-mK_X)) \hookrightarrow \OO_Y(mK_Y).$$
Lemma \ref{can1} now implies the claim.
\end{proof}

Since $K_Y + f^*(-K_C +\Gamma') \geq K_{Y/C}$, as in Remark \ref{boundary}, if we denote by $S_0$ the negative section, we obtain that:
$$\mbox{val}_{S_0} (K_{Y/C})= \inf \{ \mbox{ord}_{S_0}(K_Y + f^*(-K_C +\Gamma')) | (-K_C + \Gamma') \mbox{ is $\R$-Cartier}, \Gamma' \geq 0\},$$
with $\Gamma' \equiv C_{\Delta'}$, where $\Delta' \equiv rL + K_S$.
So, if 
$$t= \inf \{ r \in \R | \exists \Delta' \geq 0, -K_S + \Delta' \equiv  rL\}$$ 
then
$$\mbox{val}_{S_0}(K_{Y/C})= t - 1.$$
As before, we want to control for which values $r$, $\Delta'$ is numerically equivalent to an effective class. In this case $\Delta'\equiv  rL + K_S \sim \OO_{\mathbb{P}^1}(2r-2)\boxtimes \mathscr{A}^{\otimes 2r}$, hence
$$\Delta' \geq 0  \Leftrightarrow \quad r \geq 1 $$ and in particular, $\mbox{val}_{S_0}(K_{Y/C})= 0 $, and so $C$ is canonical.
\end{proof}
\vskip .3cm
Next we will show that, if $X$ is canonical and  $\mathscr{R}_X(K_X)$ is finitely generated, then $X$ has a canonical model with canonical singularities.

Let us introduce an useful Lemma from \cite[Lemma 6.2]{MR1658959}:
\begin{lem} \label{KM}Let $Y$ be a normal algebraic variety and $B$ a Weil divisor on $Y$. The following are equivalent.
\begin{enumerate}
\item $\mathscr{R}_Y(B)$ is a finitely generated sheaf of $\OO_Y$-algebras.
\item There exists a projective birational morphism $\pi: Y^+ \to Y$ such that $Y^+$ is normal, $\mbox{Ex}(\pi)$ has codimension at least $2$, $B'=\pi_*^{-1}B$ is $\Q$-Cartier and $\pi$-ample over $Y$, where $Y^+:=\mbox{Proj}_Y \sum_{m\geq 0}\OO_Y(mB)$.

$\pi: Y^+ \to Y$ is the unique morphism with the above properties.
\end{enumerate}
\end{lem}

\begin{pro} \label{fgcr} Let $X$ be a normal quasi-projective variety with canonical singularities whose canonical ring $\mathscr{R}_X(K_X)$ is a finitely generated $\OO_X$-algebra. Then the relative canonical model $X_{can}:=\mbox{Proj}_X(\mathscr{R}_X(K_X))$ exists and it has canonical singularities.
\end{pro}

\begin{proof} Since $X$ is canonical, by \cite[Proposition 8.2]{MR2501423}, we know that for any sufficiently high log resolution $f:Y \to X$, we have $K_Y - \frac{1}{m}f^{\natural}(-mK_X)\geq 0$.\\
By Lemma \ref{KM} there exists a small birational morphism $\pi: X^+ \to X$ such that $K_{X^+}$ is a relatively ample $\Q$-Cartier divisor. Also, for this morphism we have that $\pi^{\natural}(-mK_X)= -mK_{X^+}$. Let us now consider $f: Y\to X$ and $g: Y \to X^+$, a common log resolution of both $X$ and $X^+$. \\
Let us consider the map $\OO_{X^+}\cdot \OO_X(mK_X) \to  \OO_{X^+}(mK_{X^+}) $. Since $\pi_*^{-1}(K_X)= K_{X^+}$ is $\pi$-ample, $\OO_{X^+}(mK_{X^+})$ is globally generated over $X$ for $m$ sufficiently divisible, hence we have an isomorphism of sheaves. Thus
$$K_Y - g^*(K_{X^+})=K_Y + \frac{1}{m}g^*(-mK_{X^+})=K_Y + \frac{1}{m} f^{\natural}(-mK_X) \geq 0$$
 where the last equality holds by \cite[Lemma 2.7]{MR2501423}. Therefore the canonical model $X^+$ has canonical singularities.
\end{proof}

\section{Accumulation points for jumping numbers}
In this last section we use definitions and results from \cite{MR2501423}.

\vskip .3cm
Given an effective pair $(X, Z)$, we want to consider a family of ideal sheaves in the form $$\I_k=\{\I(X, t_k \cdot Z)\}_k$$
for $k\in \N$, $t_k >0$.

If $t_k$ is a decreasing sequence, then $\I_k \subset \I_{k+1}$ and by the Noetherian property, the sequence stabilizes.

If we consider an increasing sequence $t_k$, then $\I_k \supset \I_{k+1}$ and the ascending chain condition does not apply. 
We will show that (under appropriate hypothesis) even in this case the set of ideals stabilizes. Thus there are no accumulation points for the jumping numbers of the pair $(X,Z)$. We will use the following.

\begin{lem} \label{regen} Let $X$ be a projective variety and $I = \{\I_k\}_k$ the family of ideals defined above. If there exists  a line bundle $\mathscr{L}$ on $X$ such that $\mathscr{L}\otimes \I_k$ is globally generated for all $k$, then it is not possible to have an infinite sequence of ideal sheaves $\I_r \subseteq I$ such that $$\OO_X \supseteq \dots \supseteq \I_r \supsetneq  \I_{r+1} \supsetneq \I_{r+2}\supsetneq  \dots \, .$$ 
\end{lem}

\begin{proof}
Tensoring by $\mathscr{L}$ and considering cohomology we would have 
$$0 \leq \cdots \lneq h^0(\mathscr{L}\otimes \I_{r+1}) \lneq h^0(\mathscr{L}\otimes \I_{r}) \leq h^0(\mathscr{L}) = n.$$
This is impossible.
\end{proof}

The following is the main result of this section.

\begin{teo} \label{teojn} If $(X, Z)$ is an effective pair with $X$ a projective normal variety such that $X$ has either log terminal or isolated singularities.
Then the set of jumping numbers has no accumulation points, that is, given any sequence $\{t_i\}_{i\in \N}$ such that $t_i >0$ and $\lim_{i\to \infty}t_i= t$, then $$\bigcap_{i} \I(X,t_i\cdot Z) = \I(X, t_{i_0} \cdot Z)$$ for some $i_0>0$.
\end{teo}

We will need the following results.

\begin{teo} \cite[Corollary 5.8]{MR2501423} \label{cotc5.8} Let $(X, Z)$ be an effective pair, where $X$ is a projective normal variety and $Z= \sum a_k \cdot Z_k$. Let $m\geq 2$ be an integer such that $\I(X, Z)= \I_m(X, Z)$, and let $\Delta$ be an $m$-compatible boundary for $(X,Z)$. For each $k$, let $B_k$ be a Cartier divisor such that $\OO_X(B_k)\otimes \I_{Z_k}$ is globally generated, where $\I_{Z_k}$ is the ideal sheaf of $Z_k$, and suppose that $L$ is a Cartier divisor such that $L- (K_X + \Delta + \sum a_k B_k)$ is nef and big. Then $$H^i(\OO_X(L)\otimes\I(X, Z))=0 \quad \mbox{for} \; i>0.$$
\end{teo}

\begin{co} \cite[Corollary 5.9]{MR2501423} \label{cotc5.9} With the same notation and assumptions as in Theorem \ref{cotc5.8}, let $A$ be a very ample Cartier divisor on $X$. Then the sheaf $\OO_X(L + kA)\otimes \I(X, Z)$ is globally generated for every integer $k \geq \mbox{dim} X + 1$.
\end{co}

\begin{pro} \label{genweil} Let $X$ be a projective normal variety that has either log terminal or isolated singularities. Then, for any divisor $D \in \text{WDiv}_{\Q}(X)$, there exists a very ample divisor $A$ such that $\OO_X(mD)\otimes \OO_X(A)^{\otimes m}$ is globally generated for every $m\geq 1$. 
\end{pro}
\begin{proof} If $X$ has log terminal singularities, then by Remark \ref{fgcr} $\mathscr{R}_X(D)$ is a finitely generated $\OO_X$-algebra. It is then easy to see that the proposition holds. 

Let us then assume that $X$ has isolated singularities. We may assume $D \in \mbox{WDiv}(X)$. 
Let us fix a log resolution $f:Y \to X$ of $(X,D)$, where $\OO_Y \cdot \OO_X(D)= \OO_Y(\tilde{D} + F)$, with $\tilde{D}= f^{-1}_*D$ and  $F$ an exceptional divisor. Let $B$ be a general very ample divisor on $X$ such that $\OO_X(D+B)$ and $\OO_X(-K_X + B)$ are globally generated, with $\OO_Y\cdot \OO_X(-K_X + B)= \OO_Y(G)$. Then $\tilde{B}=f^*B$ and $\OO_Y(\tilde{B} +m \tilde{D} + mF)$ is globally generated, hence nef and big, for every $m>0$. 
By the Kawamata-Viehweg vanishing, if $\mathcal{G}=\OO_Y(K_Y + m \tilde{B} +m \tilde{D} +mF + G)$, $R^if_*(\mathcal{G})=0$ for all $i>0$, hence $H^i(Y, \mathcal{G})\cong H^i(X, f_* \mathcal{G}) =0$ for all $i>0$. 
Then, by Mumford regularity, we may assume that $\mathcal{F}:= f_*\OO_Y(K_Y + m((n \tilde{B} + \tilde{D} +F) +G))$ is globally generated for all $m>0$. Since $(f_*\mathcal{F})^{\vee \vee} \cong \OO_X(K_X + mD + mnB +B -K_X) \cong \OO_X( mD + (mn + 1)B)$, we have an induced short exact sequence:
$$0 \to f_*\mathcal{F} \to \OO_X(mD + (mn +1)B) \to Q \to 0$$ 
where the quotient $Q$ is suppoted on points and hence globally generated, therefore $mD +(mn +1)B$ is globally generated for all $m$. In particular $mD +m(n +1)B$ is globally generated for every $m$.
\end{proof}

\begin{re} It seems that it is not known if Proposition \ref{genweil} holds for any divisor $D \in \mbox{WDiv}_{\Q}(X)$ on any projective normal variety (regardless of the singularity). We conjecture that this is the case. Note that by Proposition \ref{genweil} this conjecture holds for surfaces.
\end{re}

\vskip .3cm
We can now prove Theorem \ref{teojn}.

\proof[Proof of Theorem \ref{teojn}]  We follow the proof of  \cite[Theorem 5.4]{MR2501423}. Let us consider an effective divisor $D$ such that $K_X -D$ is Cartier. By Proposition \ref{genweil} we know that there exists an ample line bundle $\mathscr{A}$ such that $$\mathscr{A}^{\otimes m}\otimes\mathcal{O}_X(-mD)$$ is globally generated for all $m\geq 0$.\\
For a general element $G$ in the linear system $|\mathscr{A}^{\otimes m} -mD|$, let $G = M + mD$ and we can choose $\Delta_m := \frac{1}{m}M$ as our boundary.\\
Let $B_k$ be Cartier divisors such that $\OO_X(B_k)\otimes \I_{Z_k}$ is globally generated. As in Corollary \ref{cotc5.8}, let $H$ be an ample Cartier divisor such that $H - (K_X - D + \sum a_k \cdot B_k)$ is nef and big. Then the Cartier divisor $(\mathscr{A} + H)$ is such that $$(\mathscr{A} + H) - (K_X + \Delta_m + \sum a_k B_k)$$ is nef and big for all $m$.\\
Let $B$ be a very ample Cartier divisor on $X$. Then for $\mathscr{L}:=\mathcal{O}_X(\mathscr{A} +H + sB)$, with $s > \mbox{dim} X$, we have that $$\mathscr{L}\otimes \mathcal{I}_k(X, Z)$$ is globally generated for all $k$.




By Lemma \ref{regen}, $$\bigcap_{i} \I(X,t_i\cdot Z) = \I(X, t_{i_0} \cdot Z)$$ for some $i_0>0$ and the theorem is proved.

\endproof

\addcontentsline{toc}{chapter}{Bibliography}
\nocite{*}

\bibliographystyle{alpha}  
\bibliography{bibliografia}

\begin{thebibliography}{BSTZ09}

\bibitem[BMS08]{MR2492440}
Manuel Blickle, Mircea Musta{\c{t}}{\v{a}}, and Karen~E. Smith.
\newblock Discreteness and rationality of {$F$}-thresholds.
\newblock {\em Michigan Math. J.}, 57:43--61, 2008.
\newblock Special volume in honor of Melvin Hochster.

\bibitem[BSTZ09]{Blickle:arXiv0906.4679}
Manuel Blickle, Karl Schwede, Shunsuke Takagi, and Wenliang Zhang.
\newblock Discreteness and rationality of {$F$}-jumping numbers on singular
  varieties.
\newblock {\em arXiv.org:0906.4679}, 2009.

\bibitem[CEL01]{MR1866486}
Steven~Dale Cutkosky, Lawrence Ein, and Robert Lazarsfeld.
\newblock Positivity and complexity of ideal sheaves.
\newblock {\em Math. Ann.}, 321(2):213--234, 2001.

\bibitem[dFH09]{MR2501423}
Tommaso de~Fernex and Christopher~D. Hacon.
\newblock Singularities on normal varieties.
\newblock {\em Compos. Math.}, 145(2):393--414, 2009.

\bibitem[Gro61]{MR0217084}
Alexander Grothendieck.
\newblock \'{E}l\'ements de g\'eom\'etrie alg\'ebrique. {II}. \'{E}tude globale
  \'el\'ementaire de quelques classes de morphismes.
\newblock {\em Inst. Hautes \'Etudes Sci. Publ. Math.}, (8):222, 1961.

\bibitem[Har77]{MR0463157}
Robin Hartshorne.
\newblock {\em Algebraic geometry}.
\newblock Springer-Verlag, New York, 1977.
\newblock Graduate Texts in Mathematics, No. 52.

\bibitem[HY02]{Hara:math0211008}
Nobuo Hara and {Ken-ichi} Yoshida.
\newblock A generalization of tight closure and multiplier ideals.
\newblock {\em arXiv.org:math/0211008}, 2002.

\bibitem[Kaw86]{MR846364}
Yujiro Kawamata.
\newblock On the crepant blowing-ups of canonical singularities and its
  application to degenerations of surfaces.
\newblock {\em Proc. Japan Acad. Ser. A Math. Sci.}, 62(3):104--107, 1986.

\bibitem[KLZ07]{katzman-2007}
Mordechai Katzman, Gennady Lyubeznik, and Wenliang Zhang.
\newblock On the discreteness and rationality of {$F$}-jumping coefficients,
  (v2).
\newblock {\em arXiv.org:0706.3028}, 2007.

\bibitem[KM98]{MR1658959}
J{\'a}nos Koll{\'a}r and Shigefumi Mori.
\newblock {\em Birational geometry of algebraic varieties}, volume 134 of {\em
  Cambridge Tracts in Mathematics}.
\newblock Cambridge University Press, Cambridge, 1998.
\newblock With the collaboration of C. H. Clemens and A. Corti, Translated from
  the 1998 Japanese original.

\bibitem[Kol96]{MR1440180}
J{\'a}nos Koll{\'a}r.
\newblock {\em Rational curves on algebraic varieties}, volume~32 of {\em
  Ergebnisse der Mathematik und ihrer Grenzgebiete. 3. Folge. A Series of
  Modern Surveys in Mathematics [Results in Mathematics and Related Areas. 3rd
  Series. A Series of Modern Surveys in Mathematics]}.
\newblock Springer-Verlag, Berlin, 1996.

\bibitem[Kol08]{oai:arXiv.org:0809.2579}
J{\'a}nos Koll{\'a}r.
\newblock Exercises in the birational geometry of algebraic varieties.
\newblock {\em arXiv.org:0809.2579}, October~21 2008.
\newblock Comment: Oct.21: many small corrections.

\bibitem[Laz04a]{MR2095471}
Robert Lazarsfeld.
\newblock {\em Positivity in algebraic geometry. {I}}, volume~48 of {\em
  Ergebnisse der Mathematik und ihrer Grenzgebiete. 3. Folge. A Series of
  Modern Surveys in Mathematics [Results in Mathematics and Related Areas. 3rd
  Series. A Series of Modern Surveys in Mathematics]}.
\newblock Springer-Verlag, Berlin, 2004.
\newblock Classical setting: line bundles and linear series.

\bibitem[Laz04b]{MR2095472}
Robert Lazarsfeld.
\newblock {\em Positivity in algebraic geometry. {II}}, volume~49 of {\em
  Ergebnisse der Mathematik und ihrer Grenzgebiete. 3. Folge. A Series of
  Modern Surveys in Mathematics [Results in Mathematics and Related Areas. 3rd
  Series. A Series of Modern Surveys in Mathematics]}.
\newblock Springer-Verlag, Berlin, 2004.
\newblock Positivity for vector bundles, and multiplier ideals.

\bibitem[Sch09]{schwede-2009}
Karl Schwede.
\newblock Test ideals in non-$\mathbb{Q}$-gorenstein rings, 2009.

\bibitem[Smi00]{MR1808611}
Karen~E. Smith.
\newblock The multiplier ideal is a universal test ideal.
\newblock {\em Comm. Algebra}, 28(12):5915--5929, 2000.
\newblock Special issue in honor of Robin Hartshorne.

\end{thebibliography}

\end{document}